\documentclass[a4paper, reqno]{amsart}
\usepackage{amssymb,amsmath, graphicx}
\usepackage{mathdots}
\usepackage{mathrsfs}
\usepackage{wrapfig}
\usepackage{enumerate}
\usepackage{url}
\usepackage{hyperref}
\usepackage[utf8]{inputenc}
\usepackage{tikz}
\usetikzlibrary{matrix, arrows.meta}
\usetikzlibrary{cd}
\usepackage{mathtools}

\def\qmod#1#2{{\hbox{}^{\displaystyle{#1}}}\!\big/\!\hbox{}_{
\displaystyle{#2}}}

\def\resto#1#2{{
#1\hskip 0.4ex\vline_{\hskip 0.2ex\raisebox{-0,2ex}
{{${\scriptstyle #2}$}}}}}

\def\textmap#1{\mathop{\vbox{\ialign{
                                  ##\crcr
      ${\scriptstyle\hfil\;\;#1\;\;\hfil}$\crcr
      \noalign{\kern 1pt\nointerlineskip}
      \rightarrowfill\crcr}}\;}}
\def\bigtextmap#1{\mathop{\vbox{\ialign{
                                  ##\crcr
      ${\hfil\;\;#1\;\;\hfil}$\crcr
      \noalign{\kern 1pt\nointerlineskip}
      \rightarrowfill\crcr}}\;}}

\def\textlmap#1{\mathop{\vbox{\ialign{
                                  ##\crcr
      ${\scriptstyle\hfil\;\;#1\;\;\hfil}$\crcr
      \noalign{\kern-1pt\nointerlineskip}
      \leftarrowfill\crcr}}\;}}

\def\B{{\mathbb B}}
\def\C{{\mathbb C}}

\def\H{{\mathbb H}}

\def\M{{\mathbb M}}
\def\N{{\mathbb N}}

\def\R{{\mathbb R}}
\def\Z{{\mathbb Z}}
\def\g{{\mathfrak g}}
\def\hg{{\mathfrak h}}

\def\jg{{\mathfrak j}}
\def\kg{{\mathfrak k}}

\def\pg{{\mathfrak p}}

\def\sg{{\mathfrak s}}

\def\Jg{{\mathfrak J}}

\theoremstyle{remark}

\theoremstyle{plain}

\newtheorem{sz}{Satz}[section]
\newtheorem{thry}[sz]{Theorem}
\newtheorem{pr}[sz]{Proposition}

\newtheorem{co}[sz]{Corollary}
\newtheorem{dt}[sz]{Definition}

\theoremstyle{remark}
\newtheorem{re}[sz]{Remark} 

\theoremstyle{plain}

\def\End{\mathrm {End}}

\def\U{\mathrm{U}}

\def\SU{\mathrm {SU}}
\def\SO{\mathrm {SO}}
\def\PU{\mathrm {PU}}
\def\GL{\mathrm {GL}}

\def\PSL{\mathrm {PSL}}

\def\S{\mathrm {S}}

\def\su{\mathrm {su}}

\def\Iso{\mathrm {Iso}}

\def\Hom{\mathrm{Hom}}

\def\vol{\mathrm{vol}}

\def\id{ \mathrm{id}}
\def\Im{\mathrm{Im}}
\def\im{\mathrm{im}}

\def\ad{\mathrm {ad}}

\def\U2{\mathrm{U(2)}}
\def\niq{=\kern-.18cm /\kern.08cm}
\def\ad{\mathrm{ad}}

\newcommand{\cal}{\mathcal}

\makeatletter
\DeclareFontFamily{OMX}{MnSymbolE}{}
\DeclareSymbolFont{MnLargeSymbols}{OMX}{MnSymbolE}{m}{n}
\SetSymbolFont{MnLargeSymbols}{bold}{OMX}{MnSymbolE}{b}{n}
\DeclareFontShape{OMX}{MnSymbolE}{m}{n}{
    <-6>  MnSymbolE5
   <6-7>  MnSymbolE6
   <7-8>  MnSymbolE7
   <8-9>  MnSymbolE8
   <9-10> MnSymbolE9
  <10-12> MnSymbolE10
  <12->   MnSymbolE12
}{}
\DeclareFontShape{OMX}{MnSymbolE}{b}{n}{
    <-6>  MnSymbolE-Bold5
   <6-7>  MnSymbolE-Bold6
   <7-8>  MnSymbolE-Bold7
   <8-9>  MnSymbolE-Bold8
   <9-10> MnSymbolE-Bold9
  <10-12> MnSymbolE-Bold10
  <12->   MnSymbolE-Bold12
}{}

\let\llangle\@undefined
\let\rrangle\@undefined
\DeclareMathDelimiter{\llangle}{\mathopen}%
                     {MnLargeSymbols}{'164}{MnLargeSymbols}{'164}
\DeclareMathDelimiter{\rrangle}{\mathclose}%
                     {MnLargeSymbols}{'171}{MnLargeSymbols}{'171}
\makeatother
 \def\psp#1#2%
  {\mathop{}%
   \mathopen{\vphantom{#2}}^{#1}%
   \kern-\scriptspace%
    \hskip -0.3mm{#2}} 
\begin{document}

\title[Locally homogeneous connections over  Riemann surfaces] {Locally homogeneous connections on principal bundles over hyperbolic Riemann surfaces} 
\author{Arash Bazdar }
\address{Aix Marseille Université, CNRS, Centrale Marseille, I2M, UMR 7373, 13453 Marseille, France}
\email[Arash Bazdar]{bazdar.arash@gmail.com}
\author{Andrei  Teleman}
\email[Andrei  Teleman]{andrei.teleman@univ-amu.fr}

\begin{abstract}

Let $g$ be locally homogeneous (LH) Riemannian metric on a differentiable compact manifold $M$, and $K$ be a compact Lie group endowed with an $\mathrm {ad}$-invariant inner product on its Lie algebra $\mathfrak{k}$.  
A connection $A$ on a principal $K$-bundle $p:P\to M$ on $M$ is locally homogeneous  if for  any two points $x_1$, $x_2\in M$ there exists an isometry   $\varphi:U_1\to U_2$ between open neighborhoods  $U_i\ni x_i$ which sends $x_1$ to $x_2$ and admits a $\varphi$-covering bundle isomorphism preserving the connection $A$.

This condition is invariant under the action of the automorphism group (gauge group) of the bundle, so the classification problem for LH connections leads to an interesting  moduli problem: for fixed objects $(M,g,K)$ as above describe geometrically the moduli space of all 	LH connections on principal $K$-bundles on $M$ (up to bundle isomorphisms).

Note that if $A$ is LH, then the associated connection metric $g_A$ on $P$ is locally homogenous, so it defines a geometric structure (in the sense of Thurston) on the total space of the bundle. Therefore this moduli problem is related to the classification of LH (geometric) Riemannian manifolds which admit a Riemannian submersion onto the given manifold $M$. 

Omitting the details, our moduli problem concerns the classification   of geometric fibre bundles over a given geometric base.  

We develop a general method for describing moduli spaces of LH connections on a given base. Using our method we give explicit descriptions of these moduli spaces when the base manifold is a hyperbolic Riemann surface $(M,g)$ and $K\in\{\mathrm {S}^1,\mathrm{ PU}(2)\}$.  The case $K=\mathrm{S}^1$ leads to a new construction of the moduli spaces of Yang-Mills $\mathrm{S}^1$-connections on hyperbolic Riemann surfaces, and the case  	$K=\mathrm{ PU}(2)$ leads to a one-parameter family of compact, 5-dimensional geometric manifolds, which we study in detail.

  \end{abstract}
\maketitle
\begin{quotation}
	
\end{quotation}

\tableofcontents

\section{Introduction}

\subsection{From locally homogeneous to homogeneous}
\label{Intro1}

Let $(M,g)$ be a Riemannian manifold, $K$ be a compact Lie group, $p:P\to M$ be a principal $K$ bundle on $M$, and $A$ connection on $P$. Fix an $\ad$-invariant inner product on the Lie algebra $\kg$ of $K$. Using these objects we obtain a Riemannian metric $g_A$ on $P$ uniquely determined by the conditions:
\begin{enumerate}
\item The restrictions of the differential  $p_{*}$ to the $A$-horizontal spaces are isometries, in particular $p$ is a Riemannian submersion.
\item The $A$-horizontal and vertical distributions are orthogonal.
\item For any $y\in P$ the infinitesimal action $\kg\mapsto T_y(P)$ is an isometric embedding. 
\end{enumerate}

This class of metrics on principal bundles, called  {\it connection metrics} by some authors \cite{FZ}, have been intensively studied in the literature. For instance in \cite{Je}, \cite{WZ} the authors study classes of metrics  $g_A$ which are Einstein.  In \cite{BK} the authors study classes of Einstein connection metrics on $\S^1$-bundles and $\S^3$-bundles over homogeneous Einstein manifolds.

 The main problem studied in this article   is the classification of {\it locally homogeneous} connection metrics.  Recall that a Riemannian metric $g$ on a differentiable manifold $M$ is    locally homogeneous if, for every two points $x_1$, $x_2\in M$ there exists   an isometry $\varphi:U_1\to U_2$ between open neighborhoods $U_i$ of $x_i$ such that $\varphi(x_1)=x_2$.

Let $(M,g)$ be a connected, compact, LH Riemannian manifold, $K$ be a connected, compact Lie group, and $p:P\to M$ be a principal $K$-bundle on $M$. 
\begin{dt} A  connection $A$ on $P$ will be called locally homogeneous (LH) if for
every two points $x_1$, $x_2\in M$ there exists open neighborhoods $U_i$ of $x_i$ respectively, an isometry $\varphi:U_1\to U_2$ such that $\varphi(x_1)=x_2$, and a $\varphi$-covering bundle isomorphism	$\Phi:P_{U_1}\to P_{U_2}$ such that $\Phi^*(A_{U_2})=A_{U_1}$.
\end{dt}
In this definition the subscript $_U$ denotes ``restriction   to $U$".
A connection $A$ as above is LH on $(M,g)$ if and only if   $(g,P\textmap{p} M,A)$ is a locally homogeneous triple on $M$ in the sense of \cite{Ba1}.    The classification of LH connections on a given  connected,  compact LH Riemannian manifold $(M,g)$   is related to the classification of compact LH  manifolds. Indeed,  the connection metric $g_A$ associated with any LH connection $A\in {\cal A}(P)$ is locally homogeneous, hence it defines a geometric structure in the sense of Thurston \cite[p. 358]{Th} on the total space $P$. Therefore the classification of  LH connections on $(M,g)$ is related to the classification of compact, {\it geometric} fiber bundles over $M$. This provides a strong motivation for studying and classifying locally homogeneous connections. \vspace{1mm}

Moreover the classification of 3-dimensional and 4-dimensional geometries (\cite{Th}, \cite{Sc}, \cite{Wa}, \cite{Fi}) shows that such ``fibered" geometries  are quite abundant. In dimension 3 the following geometries are fibered in our sense (can be defined by connection metrics associated with LH connections):  $\mathbf{E}^3$, $\mathbf{S}^3$, $\mathbf{S}^2\times \mathbf{E}^1$, $\mathbf{H}^2\times\mathbf{E}^1$, $\widetilde{\mathbf{SL}_2}$, $\mathbf{Nil}^3$.  In dimension 4 we have many interesting examples of fibered geometric manifolds, some of them possessing a complex structure compatible with the geometric structure \cite{Wa}. For instance the geometric structure of a topologically non-trivial principal elliptic bundle over a Riemann surface $Y$  can be defined by an LH connection metric on a  $T^2$-bundle over $Y$. The corresponding geometries are  $\mathbf{Nil}^3\times \mathbf{E}^1$, $\mathbf{S}^3\times \mathbf{E}^1$, $\widetilde{\mathbf{SL}_2}\times \mathbf{E}^1$. A large class of fibered geometric 4-manifolds is studied  in the articles \cite{Ue1}, \cite{Ue2} which deal with the classification  of geometric Seifert 4-manifolds.\vspace{1mm}
 \\

Our strategy for the classification of LH connections on principal bundles is inspired by a fundamental theorem of Singer \cite{Si}, which states that any LH, complete, simply connected Riemannian manifold is homogeneous. Let $(M,g)$ be a complete Riemannian manifold, and $\pi:\tilde M\to M$ be  the universal cover of $M$. Singer's theorem   shows that the metric $\tilde g:=\pi^*g$  is homogeneous. Therefore   any LH metric on $M$ is the quotient of a homogeneous metric on the universal cover $\tilde M$. A similar classification theorem holds for LH connections \cite{Ba1}, \cite{Ba2}:

\begin{thry} \label{mainIntro} Let $(M,g)$ be a compact LH Riemannian manifold,  $K$ be a compact Lie group, and $p:P\to M$ be a principal $K$-bundle on $M$. Let $\pi:\tilde M\to M$ be the universal cover of $M$, $\Gamma$ be the corresponding covering transformation group. Then, for any   LH connection $A$ on $P$ there exists
\begin{enumerate}
\item   A connection $B$ on the pull-back bundle  $Q:= \pi^*(P)$.
\item  A  closed subgroup $G\subset \Iso(\tilde M,\tilde g)$ acting transitively  on $\tilde M$ which contains $\Gamma$ and leaves  invariant the gauge class $[B]\in {\cal B}(Q)$.
\item A lift $\jg:\Gamma\to {\cal G}_G^B(Q)$ of  the inclusion monomorphism $\iota_\Gamma:\Gamma\to G$, where ${\cal G}_G^B(Q)$ stands for the group of automorphisms of the pair $(Q,B)$ which lift transformations in $G$.
\item An isomorphism between the $\Gamma$-quotient of $(Q,B)$ and the initial pair $(P,A)$.

\end{enumerate}
\end{thry}

We can interpret Theorem \ref{mainIntro} as follows: any pair $(P,A)$ consisting of a principal $K$-bundle on $M$ and a LH connection on $P$ can be obtained as the $\Gamma$-quotient of a homogeneous pair $(Q, B)$ on the universal cover $\tilde M$.

We will denote by $\hat G$ the group ${\cal G}_G^B(Q)$ given by Theorem \ref{mainIntro}  to save on notations; it is a Lie group and its definition yields the short exact sequence 
\begin{equation}\label{FirstExSeq}
\{1\}\to {\cal G}^B(Q)\to \hat G\textmap{\pg_B} G\to \{1\},
\end{equation}
whose left hand term ${\cal G}^B(Q)$ is  the stabilizer of the connection $B$ with respect to the action of the gauge group ${\cal G}(Q)$ on the space of connections ${\cal A}(Q)$. Fixing a point $y_0\in Q$ gives an isomorphism between ${\cal G}^B(Q)$ and a closed subgroup $L$ of $K$, namely the centraliser in $K$ of the holonomy group of $B$ \cite[Lemma 4.2.8]{DK}.  The  conjugacy class  of $L$ is independent of $y_0$.    
Note that  $B$ is irreducible if and only if its stabilizer  coincides with center $Z(K)$ of $K$.

The obtained pair $(Q,B)$ is $\hat G$-homogeneous \cite{BiTe}. Therefore one can use \cite[Theorem 2.10]{BiTe} which gives, in full generality, an explicit  description of the moduli space of homogeneous connections on a given homogeneous space. 

%In order to apply this idea we have to classify first the  Lie group extensions of $G$ by a group of the form  ${\cal G}^B(Q)$. This Lie-group extension classification problem is interesting, but, in general difficult. Fortunately in  the situations we consider in this article, one has either ${\cal G}^B(Q)=\{\id_Q\}$ (when $B$ is an irreducible $\PU(2)$-connection), or ${\cal G}^B(Q)=\S^1$ (when $B$ is an arbitrary $\S^1$-connection).  In the first case one has $\hat G=G$, and in the second case $\hat G$ is an $\S^1$-extension by $\S^1$, so in both cases the Lie-group extension problem mentioned above can be solved easily. We refer to \cite[Ch. 1, Section 3]{Ba2} for comments on the general case.  

In this way we obtain a general, effective method for the classification of LH connections on a given LH connected, Riemannian manifold $(M,g)$:

\begin{enumerate}
\item Determine all closed subgroups $G\subset \Iso(\tilde g,\tilde M)$ which contain  $\Gamma$ and act  transitively on the universal cover $\tilde M$. Note that, if $G$ acts transitively on $\tilde M$,  also does its identity component $G_0$. Moreover,  in many interesting situations  $G_0$ contains $\Gamma$, so one can assume that $G$ is connected.

\item For a  subgroup $G\subset \Iso(\tilde g,\tilde M)$ as above classify the pairs $(\hat G, \jg)$ consisting of a Lie group extension $\hat G$ of $G$  by a closed subgroup $L$ of $K$, and a lift $\jg:\Gamma\to\hat G$ of the monomorphism $\iota_\Gamma: \Gamma\hookrightarrow G$.  For  such a pair, classify the $\hat G$-homogeneous pairs $(Q,B)$ on $\tilde M$  using  \cite[Theorem 2.10]{BiTe}. Select the $\hat G$-homogeneous pairs $(Q,B)$ whose stabilizer is $L$, and such that $\hat G$ acts effectively on $Q$. 
\end{enumerate}

Note that, in fact, only very few conjugacy classes of subgroups $L\subset K$ may occur. Indeed, if $L$ is the stabilizer of a $K$-connection, then it coincides with the centralizer of  a subgroup $H\subset K$. This is a very restrictive condition (see \cite[Ch. 1, Section 3]{Ba2}). 
%For instance, for $K=\SO(3)$ the only possibilities (up to conjugacy) are: $\{\id\}$, $\O(1)$, $\SO(2)$, $\O(2)$, $\SO(3)$. 
For any Lie group extension $\hat G$ as above we obtain a space of triples $(Q,B,\jg)$, where $(Q,B)$ is a $\hat G$-homogeneous pair, and  $\jg:\Gamma\to\hat G$ is a lift of $\iota_\Gamma$ to $\hat G$.  Any such triple yields  a pair $(P,A)$ consisting of a principal $K$-bundle on $M$ and a LH connection   on $P$ (obtained as the $\Gamma$-quotient of $(Q,B)$).   

By Theorem  \ref{mainIntro} we know that in this way we obtain all LH connections on $K$-bundles on $M$. 
\subsection{Classification Theorems}
We will use the general method explained above  to classify  the  locally homogeneous connections on $\S^1$ and $\PU(2)$-bundles on hyperbolic Riemann surfaces.

Consider first the case $K=\S^1$.  An $\S^1$-connection on a  surface endowed with  a constant curvature metric is LH if and only if it is Yang-Mills \cite[Ch. 4, Proposition 2.1]{Ba2}.  For any   $t\in \R$ we will define:
\begin{enumerate}
\item 	A Lie group extension 
$$\{1\}\to \S^1\to\hat G_t\to \PSL(2,\R)\to\{1\} 
$$
of $ \PSL(2,\R)$ by $\S^1$.
\item A morphism $\chi_t:\hat H_t\to \S^1$, where  $\hat H_t\subset \hat G_t$ is the stabilizer of the point $x_0=i\in\H^2$,  and a $\hat G_t$-invariant connection $A_{t}$ on the associated principal $\S^1$-bundle  $\hat G_t\times_{\chi_t}\S^1\to\H^2$. 
\end{enumerate}
Using these objects we will prove the following result, which is stated without proof in \cite[Ch. 4, Proposition 2.2]{Ba2}:
\begin{thry} \label{S1classThintro} Let $M:=\H^2/\Gamma$ be a compact, hyperbolic Riemann surface, where $\Gamma$ is a discrete subgroup of $\PSL(2,\R)$.  
For a   principal $\S^1$-bundle  $P$  on $M$ put $t:= \frac{c_1(P)}{2(1-g)}$. 
For any LH connection  $A$ on $P$ there exists a unique lift $\jg:\Gamma\to \hat G_t$ of the monomorphism $\Gamma\hookrightarrow\PSL(2,\R)$ such that
$$(P,A)\simeq \qmod{(\hat G_t\times_{\chi_t}\S^1,A_{t})}{\Gamma}\,,
$$
where $\Gamma$ acts on $\hat G_t\times_{\chi_t}\S^1$ via $\jg$.   
\end{thry}

Moreover, varying the lift $\jg$ in the space $\Jg$ of $\hat G_t$-valued lifts of $\Gamma\hookrightarrow\PSL(2,\R)$, we obtain a parametrization of the torus of gauge classes of Yang-Mills $\S^1$-connections on $P$. This gives a new and explicit construction of this Yang-Mills  moduli space.

The classification of locally homogeneous $\PU(2)$-connections on a hyperbolic Riemann surface $(M,g)$ is obtained taking into account the stabilizer of the pull-back connection on the universal cover $\H^2$. The most interesting case is the one of LH connections with irreducible pull-back to $\H^2$. Let $H$ be the stabilizer of the point $x_0=i\in\H^2$ with respect to the standard $\PSL(2,\R)$-action on $\H^2$.  We will define a Lie group monomorphism $\chi_1:H\to \PU(2)$ and for any $z\in\C$, a $\PSL(2,\R)$-invariant connection $A_z$ on the associated bundle  $P_{\chi_1}:=\PSL(2,\R)\times_{\chi_1}\PU(2)$. We will prove (see Proposition 2.3 \cite{Ba2} and Theorem \ref{classPU(2)irr} in this article):

\begin{thry} \label{classPU(2)irrintro} Let $(M,g)$ be the hyperbolic compact Riemann surface, where $M=\H^2/\Gamma$ for a discrete subgroup $\Gamma\subset \PSL(2,\R)$ acting properly discontinuously on $\H^2$. Let  $P$ be a principal $\PU(2)$-bundle on $M$, and $A$ be a LH connection on $P$ whose  pull-back  to $\H^2$ is irreducible.  There exists a unique $r>0$ for which the pair  $(P,A)$ is isomorphic to the $\Gamma$-quotient of $(P_{\chi_1},A_{r})$ by $\Gamma$.\end{thry}

Therefore the moduli space of locally homogeneous $\PU(2)$-connections on $(M,g)$ (with irreducible pull-back to $\H$) can be naturally identified with $(0,\infty)$. This result yields a family of compact geometric 5-manifolds which are $\SO(3)$-bundles over  hyperbolic Riemann surfaces.

\begin{re} For any $r\in (0,\infty)$ we obtain a $\PU(2)$-invariant, LH Riemannian metric $g_r$ on the quotient bundle   $Q=P_{\chi_1}/\Gamma$, which makes the bundle projection $Q\to M$ a Riemannian submersion.
\end{re}

Our one parameter family of LH connections can be obtained explicitly in a geometric way using a foliation by umbilic surfaces of the cylinder $M\times\R$ endowed with a hyperbolic metric (see Proposition \ref{hypcyl}):
 \begin{pr}\label{hypcylintro}
There exists a hyperbolic metric on $\M:=\R\times M$ 
such that, putting $M_t:=\{t\}\times M$ and denoting by $\B\in{\cal A}(\SO(\M))$ the Levi-Civita connection of $\M$, one has
\begin{enumerate}
\item 	$M_0$ is totally geodesic and hyperbolic.
\item For any  $x\in M$ the path $t\mapsto (t,x)$ is a geodesic of $\M$. 
 \item  For  any $t\in\R$ the surface $M_t$ is   umbilic  in $\M$ and the restriction
 $$(\resto{\SO(\M)}{M_t}, \resto{\B}{M_t})$$
 is isomorphic to the $\Gamma$-quotient of $(P_{\chi_1},A_{\sinh(t)})$.
\end{enumerate}
	
\end{pr}

Therefore all locally homogeneous  $\PU(2)$-connections on $M$ with irreducible pull-back to $\H$  can be obtained, up to equivalence, by restricting the Levi-Civita connection of the hyperbolic cylinder $\M$ to the umbilic leaves $M_t$ for $t\in (0,\infty)$.

We obtain a clear geometric interpretation of our LH Riemannian 5-manifolds: $(Q,g_r)$ can be identified with the  restriction of the frame bundle $\SO(\M)$ (endowed with its Levi-Civita connection and the associated bundle metric) to the umbilic leaf $M_t$.

This result  illustrates our general principle explained above: studying moduli spaces of LH connections on principal bundles leads to new classes of compact LH Riemannian manifolds. We conclude with several general interesting  problems concerning these manifolds:

\begin{enumerate} 
\item Classify  the  LH Riemannian manifolds obtained in this way which are Einstein. 
\item What is the behavior 	of the Ricci flow with initial condition $g_0=g_A$.
\item Give alternative geometric interpretations of these Riemannian manifolds  and their moduli spaces (using appropriate generalizations of Proposition \ref{hypcylintro}).
\end{enumerate}

\section{Homogeneous connections on principal bundles}

\subsection{General theory} \label{genth}

In this section we present briefly the main result of \cite{BiTe} on the classification of $G$-homogeneous connections on a $G$-manifold. We begin with  

\begin{dt}\cite[p. 30]{Ya} \cite[Example 4, p. 165]{SaWa}
A pair $(G,H)$ consisting of a Lie group $G$ and a closed subgroup $H\subset G$ is called reductive if $\hg$ has an $\ad_H$-invariant complement in $\g$.	
\end{dt}

Note that if $H$ is compact in $G$, then $(G,H)$ is automatically reductive. 
\vspace{2mm}

Fix now a reductive pair $(G,H)$ and an $\ad_H$-invariant complement $\sg$ of $\hg$ in $\g$. Applying left translations to $\sg$ we obtain a left invariant connection $A^\sg$ on the principal $H$-bundle $G\to G/H$.

Let $\alpha:G\times X\to X$ be a transitive smooth action of a connected Lie group $G$ on a connected differentiable $n$-manifold $X$. Fixing a point $x_0\in X$ defines an identification $G/H\textmap{\simeq} X$,  where $H$ is the stabiliser of $x_0$ with respect to $\alpha$.  The map $\pi_{x_0}:G\to X$ given by $\pi_{x_0}(g)= gx_0$ is a principal $H$-bundle on $X$. We assume that the pair $(G,H)$ is reductive, and let  $\sg$ be  an $\ad_H$-invariant complement of  $\hg$ in $\g$. We denote by $A^\sg$ the corresponding left $G$-invariant connection on the bundle  $\pi_{x_0}$.

Let now $K$ be a Lie group. Our problem is the classification of triples  $(P,A,\beta)$, where $P$ is principal $K$-bundle on $X$, $A$ is a connection on $P$, and $\beta:G\times P\to P$ an $\alpha$-covering $G$-action on $P$ by  bundle isomorphisms which preserve $A$ (see \cite{BiTe}, \cite{Ba2} for details).
In the terminology of \cite{BiTe} a  triple $(P,A,\beta)$ as above is  called a  $G$-homogeneous $K$-connection on $X$.

Note that    $G$ acts in a natural way on the set of all gauge classes of $K$-connections (on principal $K$-bundles) on $X$ \cite{BiTe}.  The gauge class $[A]$ corresponding to a $G$-homogeneous $K$-connection  $(P,A,\beta)$ is obviously $G$-invariant. Unfortunately, in general, a $G$-invariant gauge class $[A]$ is not necessarily  associated with  a $G$-homogeneous $K$-connection.

Let  $\Phi_{\alpha,K}$  be the set of isomorphism classes  of $G$-homogeneous $K$-connections on $X$. Following \cite{BiTe} we put
$${\cal A}(G,H,K):=\{(\chi,\mu)\in\Hom(H,K)\times\Hom(\g/\hg,\kg):\mu\circ\ad_h=\ad_{\chi(h)}\circ \mu, \forall h\in H   \},
$$
$${\cal M}(G,H,K):=\qmod{{\cal A}(G,H,K)}{K}\,.
$$
In the second definition  the $K$-action on $\Hom(H,K)\times\Hom(\g/\hg,\kg)$ is given by 
$$k\cdot (\chi,\mu)=(\iota_k\circ \chi,\ad_k\circ \mu)\,.$$
Therefore the moduli space ${\cal M}(G,H,K)$ is the $K$-quotient of  the $K$-invariant real algebraic  variety ${\cal A}(G,H,K)\subset \Hom(H,K)\times\Hom(\g/\hg,\kg)$.
\begin{thry}\label{ThBiTe}\cite{BiTe}
Let $\alpha:G\times X\to X$ be a transitive action of $G$ on $X$. Let $H\subset G$ be the stabiliser of a fixed point $x_0\in X$, and suppose that the pair $(G,H)$ is reductive. One has a natural identification ${\cal M}(G,H,K)\to \Phi_{\alpha,K}$.

\end{thry}

The identification ${\cal M}(G,H,K)\to \Phi_{\alpha,K}$ mentioned  Theorem  \ref{ThBiTe} can be described explicitly as follows.  Given $(\chi,\mu)\in {\cal A}(G,H,K)$ we construct  a principal $K$-bundle $P_\chi$ with a natural $\alpha$-covering $G$-action by bundle isomorphisms, and a $G$-invariant connection $A_{\chi,\mu}$ on $P_\chi$. The pair $(P_\chi,A_{\chi,\mu})$ is  constructed  as follows:
\begin{itemize}

\item  $P_\chi$  is just the  principal bundle 
$$\pi_{x_0}^\chi:G\times_\chi K\to X$$
associated with the pair $(\pi_{x_0},\chi)$. This bundle comes with 
\begin{itemize}
\item 	a distinguished point $y_0:=[e_G,e_K]\in (P_{\chi})_{x_0}$.
\item  an obvious bundle morphism 
$$
\begin{tikzcd}
G\ar[dr, "\pi_{x_0}"'] \ar[rr, "\rho_\chi"]&&	P_\chi \ar[dl, "\pi_{x_0}^\chi"]\\
&X&
\end{tikzcd}
$$
of type $\chi$ mapping $e_G$ to $y_0$.
\end{itemize}
\item The connection $A_{\chi,\mu}$ is defined by
\begin{equation}\label{Achimu} 
A_{\chi,\mu}:=(\rho_\chi)_*(A^\sg)+a_\mu,
\end{equation}
where $a_\mu$ is the unique $G$-invariant tensorial $\kg$-valued 1-form on $P_\chi$ satisfying
\begin{equation}\label{alphamu}
\rho_\chi^*(a_\mu)_{e_G}=\mu.
\end{equation}
\end{itemize}

The curvature $F_{A_{\chi,\mu}}$ of the connection $A_{\chi,\mu}$ is the unique $G$-invariant tensorial $\kg$-valued 2-form on $P_\chi$ satisfying the identity
\begin{equation}\label{curvAchimu}
\rho_\chi^*(F_{A_{\chi,\mu}})_{e_G}(u,v)=-\chi_*([u,v]^{\hg})+[\mu(u),\mu(v)]-\mu([u,v]).
\end{equation}
The right hand side of (\ref{curvAchimu})  defines a skew-symmetric  bilinear 2-form
$$\Phi_{\chi,\mu}:\g\times\g\to \kg
$$
which has an interesting  interpretation: Let $\lambda_{\chi,\mu}:\g\to \kg$ be the linear form which coincides with $\chi_*$ on $\hg$ and with $\mu$ on $\sg$.  One has the identity
$$\Phi_{\chi,\mu}(u,v)=[\lambda_{\chi,\mu}(u),\lambda_{\chi,\mu}(v)]-\lambda_{\chi,\mu}([u,v]).
$$
Therefore $\Phi_{\chi,\mu}$ measures the failure of $\lambda_{\chi,\mu}$ to be a Lie algebra morphism. Formula (\ref{curvAchimu}) gives the curvature form at the point $y_0\in (P_\chi)_{x_0}$.  
\def\PSL{\mathrm{PSL}}
\def\sl{\mathrm{sl}}

\subsection{$\PSL(2,\R)$-homogeneous connections on the hyperbolic plane}
\label{Examples}

Homogeneous connections on the hyperbolic plane have been studied in \cite{Bi}, \cite{BiTe}. In this section we present briefly the main results. 

\subsubsection{An explicit description of the moduli space. The case $K$ Abelian}

Let 
$$\H^2=\{z\in \C|\ \Im(z)>0\}$$
be the hyperbolic plane endowed with its standard hyperbolic metric $g_{\H^2}$. We recall that the group of orientation preserving isometries of $(\H^2,g_{\H^2})$ can be identified with $G:=\PSL(2,\R)$. The stabiliser $H=G_{x_0}$ of $x_0=i\in \H^2$  with respect to the standard action
 $$\alpha:\PSL(2,\R)\times \H^2 \to \H^2
 $$
 coincides with the image of $\SO(2)$ in $\PSL(2,\R)$. 
Therefore
\begin{equation}\label{ht}
H=\bigg\{h_t:=\left[\left(\begin{matrix}\cos(t/2)&\sin(t/2)\\ -\sin(t/2)&\cos(t/2)\end{matrix}\right)\right]\vline\ t\in[0,2\pi]\bigg\}\subset \mathrm{PSL}(2,\R)\,.	
\end{equation}
Let $\sigma:H\to \S^1$ be the isomorphism $h_t\mapsto e^{it}$. Our (non-standard) way to parameterize the stabilizer  $H$ is justified by the following remark, which shows that $G$, regarded as a principal $\S^1$-bundle over $\H^2$ via the isomorphism $\sigma:H\to\S^1$, can be naturally identified with the frame bundle of the oriented Riemannian surface $\H^2$.
\begin{re}\label{delta}
Let $e_1^{x_0}\in T_{x_0}$ be the tangent vector defined by the first element of the standard basis of $\R^2$, and let $\S(T_{\H^2})$ be the circle bundle of $\H^2$, endowed with the $\S^1$-action defined by the complex orientation of $\H^2$. The map 
$$\delta: G\to \S(T_{\H^2}),\ \delta(\phi):=\phi_{*x_0}(e_1^{x_0})
$$
is a $\sigma$-equivariant diffeomorphism over $\H^2$. 

In other words, $\delta$ induces an isomorphism between the  principal bundle $G\times_\sigma\S^1$ and  the  principal $\S^1$-bundle $\S(T_{\H^2})$. The latter  can be obviously identified with the frame bundle $\SO(\H^2)$ of the oriented Riemannian surface $\H^2$. 
\end{re}

Identifying $\g$ with $\sl(2,\R)$, we have
$\hg=\R \left(\begin{matrix} 0&\frac{1}{2}\cr -\frac{1}{2}&0\end{matrix}\right),
$
and an $\ad_H$-invariant complement of $\hg$ in $\g$ is
$$\sg:=\left\langle \frac{1}{2}\left(\begin{matrix}0&1\cr 1 & 0\end{matrix}\right), \frac{1}{2}\left(\begin{matrix} 1&0\cr 0&-1\end{matrix}\right) \right\rangle .
$$
Put
$$b_1:=\frac{1}{2}\left(\begin{matrix}0&1\cr 1& 0\end{matrix}\right),\ b_2:=\frac{1}{2}\left(\begin{matrix} 1&0\cr 0&-1\end{matrix}\right),\ b_3:=\frac{1}{2}\left(\begin{matrix} 0& 1\cr -1&0\end{matrix}\right).
$$
We have
$$ [b_1,b_2]=-b_3,\ \  [b_2,b_3]=b_1,\ [b_3,b_1]=b_2.
$$
The action $\ad_H$ on $\sg$ is given by 
$$\ad_{h_t}(u_1b_1+u_2b_2)=(b_1,b_2)\left(\begin{matrix} \cos(t)&-\sin(t)\cr \sin(t)&\cos(t)\end{matrix}\right)\begin{pmatrix} u_1\\ u_2\end{pmatrix},
$$
so $\ad_{h_t}$ acts on the oriented plane $\sg=\langle b_1, b_2\rangle $ by a rotation of angle $t$. \\

For a  Lie group $K$   we  obtain an identification
\begin{equation}
 {\cal A}(\PSL(2,\R),H,K)=\{(\chi,\mu)\in \Hom(H,K)\times \Hom_\R(\sg,\kg)|\ 
  \mu\in \Hom_{\S^1}^{\chi\circ\sigma'}(\sg,\kg)\},	
 \end{equation}
where $\sigma':=\sigma^{-1}$, and the subspace $\Hom_{\S^1}^{\chi\circ\sigma'}(\sg,\kg)\subset \Hom(\sg,\kg)$ is defined by 
$$\Hom_{\S^1}^{\chi\circ \sigma'}(\sg,\kg):=\{\mu\in \Hom_\R(\sg,\kg)|\ \mu\circ R_{\zeta
}=\ad_{\chi\circ \sigma'(\zeta)}\mu\ \ \forall \zeta\in \S^1\}\,.
$$
In this formula we used the notation
$$R_{e^{it}} (u_1b_1+u_2b_2)=\ad_{h_t}(u_1b_1+u_2b_2)=(b_1,b_2)\left(\begin{matrix} \cos(t)&-\sin(t)\cr \sin(t)&\cos(t)\end{matrix}\right)\begin{pmatrix} u_1\\ u_2\end{pmatrix}\,.
$$

We  can identify $\Hom_\R(\sg,\kg)$ with the complexification $\kg^\C$ of $\kg$ using the isomorphism 
$$I:\Hom_\R(\sg,\kg)\to \kg^\C,\ I(\mu)=\mu(b_1)-i\mu(b_2).$$
%
%$$I(\mu\circ R_\zeta)=(\mu\circ R_\zeta)(b_1)-i(\mu\circ R_\zeta)(b_2)=\mu(\cos(t)b_1+\sin(t) b_2)-i\mu(-\sin(t)b_1+\cos(t) b_2)
%$$
%%
%$$=(\cos(t)+i\sin(t))\mu(b_1)-i(\cos(t)+i\sin(t))\mu(b_2)=(\cos(t)+i\sin(t))(\mu(b_1)-i\mu(b_2))
%$$
 %
  Using the identity $I(\mu\circ R_{\zeta})=\zeta I(\mu)$ for $\zeta\in\S^1$, and putting 
$$\kg^\C_{\chi\circ\sigma'}:=\{Z\in \kg^\C|\ \ad_{\chi\circ\sigma'(\zeta)}(Z)=\zeta Z,\ \forall \zeta\in \S^1\},$$
we obtain a further identification (see \cite{Bi}):
\begin{equation}\label{kC}
{\cal A}(\PSL(2,\R),H,K)=\{(\chi,Z)\in \Hom(H,K)\times \kg^\C|\  Z\in \kg^\C_{\chi\circ\sigma'}
 \}\,.	
\end{equation}
Note that $\kg^\C_{\chi\circ\sigma'}\subset \kg^\C$ is just the weight  space associated with the  weight $\id_{\S^1}$ of  the $\S^1$-representation $\zeta\mapsto \ad_{\chi\circ\sigma'(\zeta)}\in \GL(\kg^\C)$. Formula (\ref{kC}) combined with Theorem  \ref{ThBiTe} yields a simple description of the moduli space of isomorphism classes of $G$-homogeneous $K$-connections on the hyperbolic space for an Abelian Lie group $K$: 

\begin{re}\label{Ab}
If $K$ is Abelian, then 	$\kg^\C_{\chi\circ\sigma'}=0$, so one has identifications
$$\Hom(H,K)\times\{0\}={\cal A}(\PSL(2,\R),H,K),$$
$$ \Hom(H,K)\textmap{\simeq} {\cal M}(\PSL(2,\R),H,K)\simeq \Phi_{\alpha,K} .$$
In particular, for $K=\S^1$ one has  natural identifications 
\begin{equation}\label{K=S1}
\Z\textmap{\simeq}{\cal M}(\PSL(2,\R),H,\S^1)\textmap{\simeq} \Phi_{\alpha,\S^1}	
\end{equation}
 given explicitly by 
 $$k\mapsto [\pi^k\circ\sigma,0]\mapsto [A_{\pi^k\circ\sigma,0}],$$
  where $\pi^k:\S^1\to \S^1$ is the Lie group morphism  $\zeta\mapsto \zeta^k$. 
\end{re}
\begin{co}\label{AsLC}
The direct image $\delta_*(A^\sg)$ coincides with the Levi-Civita connection $A_{\mathrm{LC}}$ on the $\S^1$-bundle $\S(T_{\H^2})=\SO(\H^2)$.
\end{co}
\begin{proof}
Indeed, the Levi-Civita connection on the bundle  $\S(T_{\H^2})$ (endowed with the obvious $G$-action)   is obviously $G$-homogeneous. On the other hand Remark \ref{Ab} shows that this bundle has a unique $G$-homogeneous connection.
\end{proof}
\begin{re}\label{LCk}
The  Levi-Civita connection $A_{\mathrm{LC}}$ on the $\S^1$-bundle $\SO(\H^2)$ corresponds to the integer $k=1$ under the identification  given by  Remark \ref{Ab}. Therefore, the identification $\Z\textmap{\simeq}\Phi_{\alpha,\S^1}$ given by Remark \ref{Ab} has a geometric interpretation, and is also given by
\begin{equation}\label{tensor}
k\mapsto [A_{\mathrm{LC}}^{\otimes k}].	
\end{equation}
\end{re}
In  formula (\ref{tensor}) we used the following notation: for a principal $\S^1$-bundle $P$ and a connection $A\in {\cal A}(P)$, we put $P^{\otimes k}:=P\times_{\pi^k}\S^1$ and we denote by $A^{\otimes k}$ the direct image of $A$ in $P^{\otimes k}$. The operations  $P\mapsto P^{\otimes k}$, $A\mapsto A^{\otimes k}$ correspond to the standard tensor power operations $L\mapsto L^{\otimes k}$, $\nabla\mapsto \nabla^{\otimes k}$ for Hermitian line bundles and unitary connections on Hermitian line bundles. Note that the Hermitian line bundle associated with $\SO(\H^2)$ is just the tangent line bundle of the Kählerian complex curve $\H^2$, and $A_{\mathrm{LC}}$ corresponds to the Chern connection of this Hermitian holomorphic line bundle.

%The map 
%%
%$$\delta:G\to \SO(T_{\H^2}), \ \delta(g):=g_*(e_1^{x_0})
%$$
%is  $H$-equivariant?   
%$$\frac{d}{dt}|_0\frac{\cos(\theta)(i+t)-\sin(\theta)}{\sin(\theta)(i+t)+\cos(\theta)}=
%\frac{\cos(\theta)(\cos(\theta)+i\sin(\theta))-(i\cos(\theta)-\sin(\theta))\sin(\theta)}{(\cos(\theta)+i\sin(\theta))^2}=$$
%%
%$$\frac{1}{\zeta^2}=\cos(t)-i\sin(t)
%$$
%$$g_*(e_1^{x_0})=\begin{pmatrix}\cos(t)\\ -\sin(t) \end{pmatrix}
%$$
%$$g_*(e_2^{x_0})=\begin{pmatrix}\sin(t) )\\  \cos(t) \end{pmatrix}
%$$
%$$(e_1^{x_0}, e_2^{x_0})\begin{pmatrix}\cos(t) & \sin(t)\\ -\sin(t) & \cos(t)  \end{pmatrix}
%$$

\subsubsection{The case  $K=\PU(2)$}
 \label{homogSU(2)PU(2)}

Let 
$$ \theta:\S^1\to   \PU(2)=\SU(2)/\{\pm I_2\}$$
 be the standard monomorphism  $\S^1\to\PU(2)$ given   by
$$\theta(e^{it}):=\left[\left(\begin{matrix} e^{it/2} &0\\ 0&e^{-it/2}\end{matrix}\right)\right],
$$
where $[\ ]$ means class modulo $\{\pm I_2\}$. The image  of  $\theta$  is a maximal torus of  $\PU(2)$, and any monomorphism    $\S^1\to\PU(2)$  is equivalent  with $\theta$  modulo an interior automorphism of    $\PU(2)$. Moreover, any morphism  $\chi:\S^1\to \PU(2)$  is equivalent (modulo an interior automorphism) with a morphism of the form   $\theta_k:=\theta\circ \pi^k$  with $k\in\N$. Similarly, any morphism  $\chi:H\to \PU(2)$  is equivalent with a morphism of the form   $\theta_k\circ \sigma$  with $k\in\N$.

The set of weights of the representation 
$$\S^1\ni\zeta\mapsto \ad_{\theta_k(\zeta)}\in \GL(\su(2)^\C)=\GL(\sl(2,\C))$$
 is $\{\pi^l|\ l\in\{0,\pm  k\}\}$, so $\su(2)^\C_{\theta_k}=\{0\}$ for any $k\in \N\setminus\{1\}$. For any such $k$  formula (\ref{kC}) gives $\su(2)^\C_{\theta_k}=\{0\}$. The case $k=1$  is more interesting. First note that,   by Remark \ref{delta}, we have
\begin{re}\label{Ptheta1}
The bundle 	$P_{\theta_1\circ \sigma}$ can be identified with the $\PU(2)$-extension $$\SO(\H^2)\times_{\S^1}\PU(2)$$ of the frame bundle $\SO(\H^2)$.
\end{re}

 The space $\su(2)^\C_{\theta_1}$ is the complex line $\C I(\mu_0)$, where $\mu_0\in \Hom_{\S^1}^{\theta_1}(\sg,\su(2)) $ is defined by
 $$\mu_0(u_1b_1+u_2b_2):= \frac{1}{2}    \begin{pmatrix}
0& -u_1-i u_2\\ u_1 -i u_2 & 0	
\end{pmatrix}=u_1 a_2- u_2 a_1,$$
and $(a_1,a_2,a_3)$ is the standard basis of $\su(2)$:
 $$a_1:=\frac{1}{2}\begin{pmatrix}
0&i\\ i& 0	
\end{pmatrix},\ a_2:=\frac{1}{2}\begin{pmatrix}
0&-1\\ 1& 0	
\end{pmatrix},\ a_3:=\frac{1}{2}\begin{pmatrix}
i&0\\ 0& -i	
\end{pmatrix}.
 $$
 Therefore, endowing $\Hom_{\S^1}^{\theta_1}(\sg,\su(2))$ with the complex structure induced  from $\su(2)^\C_{\theta_1}$ via $I$, we obtain
 $$\Hom_{\S^1}^{\theta_1}(\sg,\su(2))=\C \mu_0.
 $$
 Using formula (\ref{curvAchimu}) and taking into account that  the tangent vector  $b_i\in\sg\subset T_eG$ is a lift  of $e_i^{x_0}$ for $i\in\{1,2\}$, we obtain the curvature form  
 $$F_{A_{\theta_1, z\mu_0}}\in A^2(\H^2,\ad( P_{\theta_1\circ \sigma}))=A^2(\H^2,G\times_H \su(2))$$
  at the point $y_0\in P_{\theta_1\circ \sigma}$:
 $$F_{A_{\theta_1, z\mu_0}}^{y_0}=(1+|z|^2)e^1_{x_0}\wedge e^2_{x_0}\  a_3,
 $$
where  $(e^1_{x_0},e^2_{x_0})$ is the dual basis of the standard basis $(e_1^{x_0},e_2^{x_0})$ of $T_{\H^2,x_0}$. Using the $G$-invariance of the connection $A_{\theta_1\circ\sigma, z\mu_0}$, this formula determines the  curvature form $F_{A_{\theta_1\circ\sigma, z\mu_0}}$ at any point.
 
 It is interesting to have explicit geometric interpretations of the oriented Euclidian rank 3 vector bundle $E=G\times_H \su(2)\to \H^2$ associated with the principal bundle $P_{\theta_1\circ \sigma}$ via the standard isomorphism $\ad:\PU(2)\to \SO(\su(2))=\SO(3)$, and  of the linear connection $\nabla^z$ on $E$ which corresponds to $A_{\theta_1\circ\sigma, z\mu_0}$.   The morphism $\theta_1\circ \sigma$ leaves invariant the direct sum decomposition $\su(2)=\langle a_1,a_2\rangle\oplus\R a_3$; the subgroup $H$ acts on the plane  $\langle a_1,a_2\rangle$ via the isomorphism $\sigma$, and acts trivially on the line $\R$. Therefore we get a direct sum decomposition
 $$E=\big(G\times_H\langle a_1,a_2\rangle\big)\oplus \underline{\R} a_3
 $$
 where $\underline{\R} a_3$ stands for the trivial line bundle of fibre $\R a_3$. Taking into account Remark \ref{Ptheta1}, the first summand can be identified with the tangent bundle $T_{\H^2}$ via a $G$-invariant isomorphism  mapping $[e_G,a_i]$ onto $e_i^{x_0}$ for $i\in\{1,2\}$. Taking into account Corollary \ref{AsLC} and formulae (\ref{Achimu}), (\ref{alphamu}) we obtain the following matrix decomposition of the linear connection $\nabla^z$:
 \begin{equation}\label{nablaz}
\nabla^{z}=\begin{pmatrix}
\nabla^{\mathrm{LC}} &  B^z\\
 -\psp{t}{B}^z  &  \nabla^0
\end{pmatrix},
 \end{equation}
where $\nabla^{\mathrm{LC}}$ is the Levi-Civita connection on $T_{\H^2}$, $\nabla^0$ is the trivial connection on the trivial line bundle $\underline{\R} a_3$, and the second fundamental form $B^z\in A^1(\H^2,T_{\H^2})$ is the $G$-invariant  $T_{\H^2}$-valued 1-form on $\H^2$ which is given at the point $x_0$ by
$$B^{re^{it}}_{x_0}=r\big(e^1_{x_0}(
\cos(t)e^{x_0}_1+  \sin(t)e^{x_0}_2)+e^2_{x_0}(-
\sin(t)e^{x_0}_1+  \cos(t)e^{x_0}_2)\big).
$$
This shows that
\begin{re}\label{Bz}
Identifying 	$A^1(\H^2,T_{\H^2})$ with $A^0(\H^2,\End_\R(T_{\H^2}))$, and endowing $T_{\H^2}$ with its natural complex structure, one has $B^z=z\id_{T_{\H^2}}$.
\end{re}

We can now give an explicit description of the moduli space 
$${\cal M}(\PSL(2,\R),H,\PU(2))\simeq \Phi_{\alpha,\PU(2)}$$
classifying  isomorphism classes of $G$-homogeneous $\PU(2)$-connections on the hyperbolic plane. The centraliser of the morphism $\theta_1\circ\sigma:H\to \PU(2)$ with respect to the adjoint action of $\PU(2)$ is $\im(\theta_1)\simeq \S^1$, and it acts with weight 1 on the complex line $\Hom_{\S^1}^{\theta_1}(\sg,\kg)$. This shows that 
$${\cal M}(\PSL(2,\R),H,\PU(2))=\{[\theta_k\circ\sigma,0]|\ k\in\N\setminus\{1\}\}\cup {\cal M}_1
\,,
$$
where 
 
$${\cal M}_1=\{[\theta_1\circ\sigma,r\mu_0]|\ r\in[0,\infty)\}\simeq [0,\infty)\,.
$$

\subsubsection{Homogeneous $\SO(3)$-connections on the hyperbolic plane, and umbilical foliations on the hyperbolic space}
\label{umbilic}

\def\sech{\mathrm{sech}}

Let 
$$
\H^3:=\{x\in\R^3|\ x_3>0\}
$$
be the hyperbolic space endowed with the standard hyperbolic metric 
$$g=\frac{1}{x_3^2}(\sum_i dx_i^2).$$
 For a point $x\in \H^3$ the $g$-normalized vectors $\bar e_i^x=x_3 e_i^x$ give an orthonormal basis of the Euclidian space $(T_x\H^3,g_x)$.  We denote by $\bar e_i\in {\cal X}(\H^3)$ the vector field $x\to  \bar e_i^x$.

For $t\in\R$ let $f^t:\H^3\to \H^3$ be the diffeomorphism of $\H^3$ induced by linear isomorphism associated with the matrix
$$F^t:=\begin{pmatrix}
1&0&\tanh(t)\\
0&1 &0\\
0&0&\sech(t)	
\end{pmatrix}.
$$
The family $(f_t)_{t\in\R}$ has an important interpretation:
\begin{re}
For any $x\in\H^3$, the curve $t\mapsto f^t(x)$ coincides with the geodesic path  $\gamma^x:\R\to \H^3$ determined by the initial conditions
$$\gamma_x(0)=x,\ \dot\gamma_x(0)=\bar e_1^{x}.
$$
\end{re}
Put
$$H_0:=\{x\in \H^3|\ x_1=0\},\ H_t:=f_t(H_0).
$$
The surface $H_t$ is the intersection of $\H^3$ with the plane defined  by the equation $x_1=\sinh(t)x_3$. It is well known that $H_0$ is a totally geodesic hypersurface of $\H^3$, whereas $H_t$ is an umbilic hypersurface of $\H^3$ for any $t\in\R$.  The family $(H_t)_{t\in\R}$ defines a foliation ${\cal F}$ of $\H^3$ by umbilic hypersurfaces.  Note that there is an interesting literature dedicated to the classification of foliations by umbilic submanifolds (see for instance \cite{Wals}). The classification theorem \cite[Theorem 2.6]{Wals} states that any isoparametric 1-codimensional umbilic foliations of a Riemannian manifold of negative curvature is either a foliation by horospheres, or a foliation by hypersurfaces equidistant from a totally geodesic hypersurface.   Our foliation ${\cal F}$ belongs to the second class.

The following proposition gives a geometric interpretation of the $G$-homogeneous connections $A_{\theta_1\circ\sigma, s\mu_0}$  obtained with formal methods:
\begin{pr}\label{GeomInt}
Let $B_{\mathrm LC}$ be the Levi-Civita connection on $\SO(\H^3)$, and let $\varphi_t:\H^2\to \H^3$ be the embedding $(x_1,x_2)\mapsto f_t(0,x_1,x_2)$ whose image is the umbilic surface $H_t$.  There exists a natural isomorphism of pairs
$$(\varphi_t^*(\SO(\H^3)),\varphi_t^*(B_{\mathrm LC}))\simeq( P_{\theta_1\circ\sigma}, A_{\theta_1\circ\sigma,\sinh(t)\mu_0}).
$$	
\end{pr}
\begin{proof}  Using the orthogonal direct sum decomposition $\resto{T_{\H^3}}{H_t}=T_{H_t}\oplus N_{H_t}$ we see that the restriction $\resto{\SO(\H^3)}{H_t}$  can be identified with the $\SO(3)$-extension of the $\SO(2)$-bundle $\SO(H_t)$. Taking into account our geometric description of the bundle $P_{\theta_1\circ\sigma}$ (see Remark \ref{Ptheta1}), we obtain an obvious identification  $\varphi_t^*(\SO(\H^3))\simeq P_{\theta_1\circ\sigma}$.

   To get the claimed identification $\varphi_t^*(B_{\mathrm LC})=A_{\theta_1\circ\sigma,\sinh(t)\mu_0}$ one  uses formula (\ref{nablaz}) and the similar matrix decomposition of the restriction $\resto{B_{\mathrm LC}}{H_t}$. It suffices to identify the second fundamental forms of the two connections. The equality follows using Remark \ref{Bz} and a direct computation of the second fundamental form of $H_t$.	
\end{proof}

Proposition \ref{GeomInt} states that: 
\begin{re}  
Via the  diffeomorphism 
$$\H^2\textmap{ \varphi_{\mathrm{argsh}(s)}} H_{\mathrm{argsh}(s)}$$
 the $G$-homogeneous connection $A_{\theta_1\circ\sigma, s\mu_0}$ can be identified with the restriction of the Levi-Civita connection $B_{\mathrm LC}\in {\cal A}(\SO(\H^3))$ to the leaf $H_{\mathrm{argsh}(s)}$ of the umbilic foliation ${\cal F}$.	
\end{re}

This also gives a geometric interpretation of the moduli space ${\cal M}_1$: it can be identified with the space of isomorphism classes of restrictions   
$$\resto{B_{\mathrm LC}}{H_{t}},\ t\in[0,\infty).$$
\section{Classification theorems for LH  connections}
\label{ClassSection}

\subsection{The case $K=\S^1$. LH and Yang-Mills connections }
\label{S1section}

The classification of locally homogeneous $\S^1$-connections on hyperbolic Riemann surfaces is more  interesting than expected. The natural idea is to write $M$ as a quotient $\H^2/\Gamma$, where $\Gamma\subset \PSL(2,\R)$ acts properly discontinuously on $\H^2$, and to consider $\Gamma$-quotients  of $\PSL(2,\R)$-homogeneous $\S^1$-connections on $\H^2$. Unfortunately, in this way one obtains a very small class of locally homogeneous $\S^1$-connections on $M$.  Indeed, Remark \ref{LCk} shows that  this method yields only the tensor powers $(A_{\mathrm{LC}}^M)^{\otimes k}$ of the Levi-Civita connection $A_{\mathrm{LC}}^M$ of $M$.  The tensor power $(A_{\mathrm{LC}}^M)^{\otimes k}$ is a connection on $\SO(M)^{\otimes k}$, whose Chern class is $2k(1-g(M))$. This shows that a principal $\S^1$-bundle $P$ admits a connection obtained in this way if and only if $c_1(P)\in 2\Z(1-g(M))$ and, if this is the case, $P$ admits a unique gauge class of such a connection.  On the other hand one has the following general result concerning the classification of LH $\S^1$-connections on  Riemann surfaces:
\begin{pr}\label{LocHomYM} \cite{Ba2}
Let $(M,g)$ be a connected, oriented, compact Riemann surface endowed with a Riemannian metric with constant curvature, let $P$ be a principal $\S^1$-bundle on $M$ and $A\in {\cal A }(P)$ be a connection on $P$. Then $A$ is LH if and only if $A$ is Yang-Mills.	
\end{pr}

This shows that, for {\it any} fixed integer $k\in \Z$, the set of isomorphism classes of locally homogeneous pairs  $(P,A)$ with $c_1(P)=k$ can be identified with the moduli space of Yang-Mills connections on a Hermitian line bundle of Chern class $k$, which is a torus of dimension $2g$ (see for instance \cite{Te}).

We will see that, using the classification Theorem \ref{mainIntro}  and the method explained in section \ref{Intro1}, any locally homogeneous (or, equivalently, any Yang-Mills) $\S^1$-connection on $M$ can be obtained explicitly as the  quotient of  a homogeneous connections on $\H^2$, but we will have to replace the group $\PSL(2,\R)$ by an $\S^1$-extension of it (which depends on the Chern class of the underlying $\S^1$-bundle).  \\

We will need the universal cover $c:\widetilde{\PSL(2,\R)}\to \PSL(2,\R)$  of the Lie group $\PSL(2,\R)$. The loop $\S^1\to \PSL(2,\R)$  
$$e^{i\tau}\mapsto h_\tau:= \left[\left(\begin{matrix}\cos(\tau/2)&\sin(\tau/2)\\ -\sin(\tau/2)&\cos(\tau/2)\end{matrix}\right)\right], $$
defines a generator $\gamma$ of $\ker(c)=\pi_1(\PSL(2,\R))$. The Lie group $\hat G_t$  defined by
$$\hat G_t:=\qmod{\widetilde{\PSL(2,\R)}\times\S^1}{\{(\gamma^k,e^{-2\pi tk i})|\ k\in\Z\}}\, .
$$
fits in the short exact sequence
$$1\to \S^1\textmap{j_t} \hat G_t\textmap{c_t}  \PSL(2,\R)\to 1 \,,
$$
where the morphisms $j_t$, $c_t$ are given by $j_t(z):=[e,z]$, $c_t([y,\zeta]):= c(y)$. 

The restriction of $c$ to  $\tilde H:= c^{-1}(H)  \subset \widetilde{\PSL(2,\R)}$ is a  universal cover of $H$.  Note that $\tilde H$ is the image of a 1-parameter subgroup $\R\ni \tau \mapsto\tilde h_\tau$, where  $\tilde h_\tau $ is a lift of $h_\tau$, and  $\gamma=\tilde h_{2\pi}$. The subgroup $\hat H_t:= c_t^{-1}(H)\subset \hat G_t$ can be written as
$$\qmod{\tilde H\times \S^1}{\{(\gamma^k,e^{-2\pi tk i})|\ k\in\Z\}}\,,
$$
and is abelian. We have a short exact sequence
\begin{equation}\label{shexseq}
1\to \S^1\textmap{i_t} \hat H_t\textmap{p_t}  H\to 1,	
\end{equation}
where the morphisms $i_t$, $p_t$ are defined similarly to $j_t$, $c_t$.  The Lie algebra $\hat\g_t$ of $\hat G_t$   decomposes as  $\hat\g_t=\sl(2,\R)\times i\R$;  the subspace $\hat\sg_t:=\sg\times\{0\}\subset \hat\g_t$ is a  $\hat H_t$-invariant complement of $\hat \hg_t$.

The group morphism  $\chi_t:\hat H_t\to \S^1$, $\chi_t([\tilde h_\tau,\zeta]):= e^{ i\tau  t}\zeta$  is  a left splitting of the short exact sequence (\ref{shexseq}), so it defines an isomorphism $\hat H_t\simeq \S^1\times \S^1$. 

The group $ \hat G_t$ acts on  $\H^2$ via $c_t$ and  the stabiliser of  $x_0$ with respect to this action is $c_t^{-1}(H)=\hat H_t$. Using the general construction method explained in section \ref{genth} we obtain a $\hat G_t$-homogeneous  connection $(P_{\chi_t},A_{{\chi_t},0})$ associated with 
the triple $(\hat G_t, \hat H_t,\hat\sg_t)$ and the pair $(\chi_t,0)$.

Using this construction we obtain the following theorem, which describes all Yang-Mills $\S^1$-connections (hence all LH $\S^1$-connections) on hyperbolic Riemann surfaces as quotients of homogeneous connections on $\H^2$. This result is stated without proof in \cite{Ba2}.
\begin{thry} \label{S1classTh} Let $M:=\H^2/\Gamma$ be a compact, hyperbolic Riemann surface, where $\Gamma$ is discrete subgroup of $\PSL(2,\R)$ acting properly discontinuously on $\H^2$.  
For a   principal $\S^1$-bundle  $P$  on $M$ put $t:= \frac{c_1(P)}{2(1-g)}$. 
For any LH connection  $A$ on $P$ there exists a unique lift $\jg:\Gamma\to \hat G_t$ of the monomorphism $\Gamma\hookrightarrow\PSL(2,\R)$ such that
$$(P,A)\simeq \qmod{(\hat G_t\times_{\chi_t}\S^1,A_{\chi_t,0})}{\Gamma},
$$
where $\Gamma$ acts on $\hat G_t\times_{\chi_t}\S^1$ via $\jg$.  
\end{thry}
\begin{proof}  
The curvature formula (\ref{curvAchimu})	 shows that
$$F_{A_{\chi_t,0}}=tF_{A_{\chi_1,0}}=tF_{A_{\sigma,0}}.
$$
On the other hand, by Remark \ref{LCk}, we have $F_{A_{\sigma,0}}=F_{A_\mathrm{LC}}$.  Note that the Chern form $c_1(A_{\chi_t,0})=\frac{i}{2\pi} F_{A_{\chi_t,0}}$ is a $\Gamma$-invariant 2-form on $\H^2$, so it descends to a 2-form $\bar c_1(A_{\chi_t,0})$ on the closed surface $M$.  
Putting $c:=c_1(P)\in H^2(M,\Z)\simeq\Z$, $t:=\frac{c}{2(1-g)}$ we obtain 
\begin{equation}\label{Zf}
[\bar c_1(A_{\chi_t,0})]_{\rm DR}=c. 	
\end{equation}

On the other hand the obstruction to the existence of a lift $\jg:\Gamma\to \hat G_t$ of the embedding monomorphism $\iota:\Gamma\hookrightarrow \PSL(2,\R)$ is a cohomology class $o\in H^2(\Gamma,\S^1)=H^2(M,\S^1)$. Using a standard \v{C}ech-de Rham double complex argument, it follows that $o$ is precisely the image of $[\bar c_1(A_{\chi_t,0})]_{\rm DR}\in H^2(M,\R)$ in the quotient $H^2(M,\S^1)=H^2(M,\R)/H^2(M,\Z)$. Therefore (\ref{Zf}) shows that $o$ vanishes, so   the set $\Jg$ of 
 lifts $\jg:\Gamma\to \hat G_t$ of the embedding monomorphism $\iota:\Gamma\hookrightarrow \PSL(2,\R)$ is non-empty. For any $\jg\in\Jg$ the Chern class of the  corresponding quotient bundle 
 $$P_\jg=P_{\chi_t}/_\jg\ \Gamma$$
  is  $[\bar c_1(A_{\chi_t,0})]_{\rm DR}=c$, so  $P_\jg\simeq P$. Therefore for any $\jg\in\Jg$ we obtain a quotient connection $A_\jg$ induced by $A_{\chi_t,0}$ on an $\S^1$-bundle $P_\jg\simeq P$. The set $\Jg$ has an obvious structure of a $\Hom(\Gamma,\S^1)$-torsor, and it is easy to see that, for  any $\rho\in \Hom(\Gamma,\S^1)$, the connection $A_{\rho\jg}$  can be identified with the tensor product of $A_\jg$ by the flat connection associated with $\rho$. Therefore the set of isomorphism classes $\{[A_\jg]|\ \jg\in\Jg\}$ coincides with the whole torus of gauge classes of Yang-Mills connections on $P$, and the map $\jg\mapsto [A_\jg]$ is bijective.
\end{proof}

\subsection{The case $K=\PU(2)$}  

Let $\Gamma\subset \PSL(2,\R)$ be a discrete subgroup acting properly discontinuously on  $\H^2$ with  compact quotient, and let $(M,g_M)$ be the hyperbolic  Riemann surface $M:=\H^2/\Gamma$. The classification of    locally homogeneous $\PU(2)$-connections on $M$ is obtained by taking into account the stabilizer of the pull-back connection on $\H^2$. Let  $P$ be a principal $\PU(2)$-bundle on $M$, and $A$ be a locally homogeneous connection on $P$. Let $B$ be the pull-back connection on the pull-back bundle $Q$ on $\H^2$.

\subsubsection{Locally homogeneous  connections with irreducible pull-back}

Suppose that $B$ is irreducible, i.e. it  has 	trivial stabiliser. In this case the classification Theorem \ref{mainIntro} applies with ${\cal G}^B(Q)=\{\id\}$, and gives 
\begin{enumerate}
\item   A  closed subgroup $G\subset \Iso(\H^2, g_{\H^2})$ acting transitively  on $\H^2$ which contains $\Gamma$ and leaves  invariant the gauge class $[B]\in {\cal B}(Q)$.
\item A lift $\jg:\Gamma\to G$ of  the inclusion monomorphism $\iota_\Gamma:\Gamma\to G$.
\item An isomorphism between the $\Gamma$-quotient of $(Q,B)$ and the initial pair $(P,A)$.
\end{enumerate}

The connected component $G_0\subset G$ of $\id\in G$ still acts transitively on $\tilde M$ but it is not clear if it still contains $\Gamma$. 

\begin{re} \label{unimodular} With the notations introduced in Theorem \ref{mainIntro} suppose that $(\tilde M,\tilde g)=(\H^2,g_{\H^2})$, and put $\Gamma_0:=\Gamma\cap G_0$. The quotient $\Gamma_0\backslash G_0$ can be identified with the connected component of  $\Gamma\in \Gamma\backslash G$ in this quotient, therefore the compactness of $\Gamma\backslash G$ implies the compactness of $\Gamma_0\backslash G_0$. This shows that $G_0$ is a connected Lie subgroup of $\PSL(2,\R)$  which acts transitively on $\H^2$ and is unimodular. This implies $G_0=\PSL(2,\R)$, in particular  $G_0$ contains $\Gamma$.
\end{re}

Therefore we may suppose $G=G_0=\PSL(2,\R)$, so we reduced our problem to the classification of $\PSL(2,\R)$-homogeneous $\PU(2)$-connections on $\H^2$ studied in section \ref{homogSU(2)PU(2)}. We obtain:
\begin{thry} \label{classPU(2)irr} Let $\Gamma\subset \PSL(2,\R)$ be a discrete subgroup acting properly discontinuously on  $\H^2$ with compact quotient, and let $(M,g_M)$ be the hyperbolic  Riemann surface $M:=\H^2/\Gamma$.
Let  $P$ be a principal $\PU(2)$-bundle on $M$, and $A$ be a  locally homogeneous connection on $P$ whose  pull-back connection to $\H^2$ is irreducible.  There exists a unique $r>0$ for which the pair  $(P,A)$ is isomorphic to the   $\Gamma$-quotient of $(P_{\theta_1\circ\sigma},A_{\theta_1\circ\sigma,r\mu_0})$.\end{thry}
Therefore 
\begin{co} The moduli space of locally homogeneous  $\PU(2)$-connections on $M$ with irreducible pull-back to $\H^2$ can be naturally identified to the half-line $(0,\infty)$.	
\end{co}

Using the results proved in section \ref{umbilic} we obtain a geometric interpretation of this 1-parameter family  of locally homogeneous  $\PU(2)$-connections on $M$ in terms of umbilic foliations on hyperbolic 3-manifolds:
\begin{pr}\label{hypcyl}
There exists a hyperbolic metric $g_\M$ on $\M:=\R\times M$ 
such that, putting $M_t:=\{t\}\times M$ and denoting by $\B\in{\cal A}(\SO(\M))$ the Levi-Civita connection of $\M$, one has
\begin{enumerate}
\item 	$M_0$ is totally geodesic and hyperbolic.
\item For any  $x\in M$ the path $t\mapsto (t,x)$ is a geodesic of $\M$. 
 \item  For  any $t\in\R$ the surface $M_t$ is   umbilic  in $\M$ and the restriction
 $$(\resto{\SO(\M)}{M_t}, \resto{\B}{M_t})$$
 is isomorphic to the $\Gamma$-quotient of $(P_{\theta_1\circ\sigma},A_{\theta_1\circ\sigma,\sinh(t)\mu_0})$.
\end{enumerate}
	
\end{pr}
\begin{proof}
Using the notations introduced in section \ref{umbilic}, note  that for any orientation preserving isometry $\psi$ of $\H^2$ there exists an orientation preserving isometry $\Psi$ of $\H^3$ such that  $\Psi\circ\varphi_t=\varphi_t\circ \psi$ for any $t$. The isometry $\Psi$ is obtained by extending $\psi$ in the obvious way along the geodesics $\gamma^x$, $x\in \H^2$. The obtained map $\psi\mapsto \Psi$ is a Lie group monomorphism.  Let $\hat\Gamma$ be the image of $\Gamma$ under this monomorphism, and put $\hat M:=\H^3/\hat\Gamma$. We obtain a diffeomorphism
$$\hat\varphi:\R\times M\to \hat M  
$$
induced by the diffeomorphism $\R\times\H^2\ni (t,x)\to \varphi_t(x)\in\H^3$. Putting $g_\M:=\hat\varphi^*(g_{\H^3})$	the claim follows.
\end{proof}
Proposition \ref{hypcyl} shows that
\begin{re}
 All locally homogeneous  $\PU(2)$-connections on $M$ with irreducible pull-back to $\H^2$  can be obtained, up to equivalence, by restricting the Levi-Civita connection of the hyperbolic cylinder $\M$ to the umbilic leaves $M_t$ for $t\in (0,\infty)$.	
\end{re}

\subsubsection{Locally homogeneous  connections with Abelian, non-flat pull-back}

Suppose now that the stabilizer of $B$ is $\S^1$. Therefore the holonomy of $B$ at a point $y\in Q$ is contained in the centralizer of $\S^1$ in $\S^1$ (which is $\S^1$), so $Q$ admits a $B$-invariant $\S^1$-reduction $Q_0\subset Q$. The direct sum decomposition $\su(2)=i\R\oplus\C$ is $\S^1$-invariant; the subgroup $\S^1\subset\PU(2)$ acts trivially on the first summand, and with weight 1 on the second. Therefore, we obtain a $B$-invariant direct sum decomposition
\begin{equation}\label{dsdec} 
\ad(Q)=Q_0\times_{\S^1}\su(2)=i\underline{\R}\oplus L	
\end{equation}
where $L$ is a Hermitian line bundle on $\H^2$. The curvature form $F_B$ takes values in the first summand, so   $F_B=\vol_{\H^2}\eta_B$ with $\eta_B\in {\cal C}^\infty(\H^2,i\R)$. 

The local homogeneity condition of $A$ shows that $F_B$ has constant norm, so $\eta_B$ is constant. 

Since $B$ is not flat, it follows that $\eta_B$ is a non-vanishing constant. The decomposition  (\ref{dsdec}) is $B$-parallel, so $\eta_B$ is also a $B$-parallel section of $\ad(Q)$. The group $\Gamma$ acts by $B$-preserving bundle isomorphisms of $Q$ and by orientation preserving isometries on $\H^2$, so it leaves invariant $F_B$ and  $\vol_{\H^2}$; it follows that it also leaves invariant the section $\eta_B\in A^0(\H^2,\ad(Q))$.  Therefore $\eta_B$ descends to a non-trivial $A$-parallel section of $\ad(P)$, which implies that $P$  admits an  $A$-parallel $\S^1$-reduction $P_0\subset P$. The obtained $\S^1$-connection $A_0\in {\cal A}(P_0)$ is still locally homogeneous, so the classification of locally homogeneous $\PU(2)$-connections of this type on $M$ reduces to the classification of   locally homogeneous $\S^1$-connections, which has been studied in section \ref{S1section}. 

\subsubsection{Locally homogeneous  connections with flat pull-back} 
If the stabiliser of $B$ is $\PU(2)$, then $A$ will be flat, so the pair $(P,A)$ will be defined by a representation $\pi_1(M)=\Gamma\to\PU(2)$.   The moduli space of isomorphism classes of flat $\PU(2)$-connections on $M$ is just the quotient $\Hom(\pi_1(M),\PU(2))/\PU(2)$ of morphisms $\pi_1(M)\to \PU(2)$ modulo conjugation. The statement generalizes for an arbitrary locally homogeneous Riemannian manifold $(M,g)$ and Lie group $K$:  any flat $K$-connection is LH, and the moduli space of isomorphism classes of flat $K$-connections on $M$ can be identified with the quotient $\Hom(\pi_1(M),K)/K$. The case when $(M,g)$ is a  Riemann surface and $K=\PU(r)$ is especially interesting because of the classical Narasimhan–Seshadri theorem \cite{NaSe}, \cite{Do} which states that a holomorphic rank $r$-bundle on a Riemann surface (regarded as a complex projective curve) is polystable if and only if it admits a projectively flat Hermitian metric. In the special case of bundles of degree 0, we get an isomorphism of moduli spaces 
$$\Hom(\pi_1(M),\SU(r))/\SU(r)\textmap{\simeq}{\cal M}^{\rm pst}(r,0)$$
onto the moduli space of  polystable holomorphic rank $r$-bundles of degree 0.

%\noindent\textbf{Acknowledgement}\textit{.} Your acknowledgement, if any.

\vskip 0,65 true cm

\end{document}